\documentclass [12pt]{amsart}
\usepackage{amssymb,amsmath,amscd,amsthm}

\begin{document}

\newtheorem{proposition}{Proposition}
\newtheorem{lemma}{Lemma}
\newtheorem{theorem}{Theorem}
\newtheorem{corollary}{Corollary}
\newtheorem*{corol}{Corollary}
\newtheorem*{proposition*}{Proposition}
\newtheorem*{lemma*}{Lemma}
\newtheorem*{theorem*}{Theorem}
\newtheorem*{corollary*}{Corollary}

\newtheorem{definition}{Definition}
\newtheorem{remark}{Remark}
\newtheorem{example}{Example}
\newtheorem*{definition*}{Definition}
\newtheorem*{remark*}{Remark}
\newtheorem*{example*}{Example}

\newcommand \Aut {\,\mathrm{Aut}\,}
\newcommand \GL {\,\mathrm{GL}\,}
\newcommand \SL {\,\mathrm{SL}\,}
\newcommand \PGL {\,\mathrm{PGL}\,}
\newcommand \PSL {\,\mathrm{PSL}\,}
\newcommand \im {\,\mathrm{Im}\,}
\newcommand \ord {\,\mathrm{ord}\,}
\newcommand \eps {\varepsilon}
\newcommand \define { \stackrel{def}\Longleftrightarrow }
\newcommand \then \Rightarrow
\renewcommand \iff \Leftrightarrow
\newcommand{\Iff}{\Longleftrightarrow}
\newcommand \Z[1] {\mathbb Z_{#1}}
\newcommand \mapson[1] {\overset{#1}\mapsto}
\newcommand \F {\mathbf F}
\newcommand \Pair[1][{}] {(\xi_{#1}, \eps_{#1})}
\newcommand \id {\,\mathrm{id}\,}
\renewcommand \leq \leqslant
\renewcommand \geq \geqslant

\def\udk{512.541.6 + 510.67}
\title{Elementary equivalence of the automorphism groups of reduced Abelian $p$-groups}

{\Large {\bf Elementary equivalence of the automorphism groups of reduced Abelian $p$-groups }}

\vskip 0.1cm \centerline{\large \large M.A.~Roizner}
\vskip 0.5cm

{\large {\bf Annotation }}
Consider unbounded reduced Abelian $p$-groups ($p \geq 3$) $A_1$ and $A_2$. In this paper, we prove that if the automorphism groups $\Aut A_1$ and $\Aut A_2$ are elementary equivalent then the groups $A_1$ and $A_2$ are equivalent in the second order logic bounded by the final rank of the basic subgroups of $A_1$ and $A_2$.

\section{Introduction}

In this paper, we consider elementary properties (i.\,e. properties expressible in the first order logic)
of the automorphism groups of reduced Abelian $p$-groups.

The first who considered connection of elementary properties of the different models with elementary properties of derivative models was A.I.\,Maltsev \cite{Maltsev} in 1961. He proved that the groups $G_n(K)$ and $G_m(L)$, where $G=\GL,\SL,\PGL,\PSL$ and $n,m\geqslant 3$, $K,L$ are fields of characteristics~$0$, are elementary equivalent iff $m=n$ and fields $K$ and $L$ are elementary equivalent.

In 1992, this theory was continued with the help of ultraproduct construction and Keisler-Chang Isomorphism Theorem by K.I.\,Beidar and A.V.\,Mikhalev in \cite{BeiMikh}, in which they found a general approach to problems of elementary equivalence of different algebraic structures and generalized Maltsev theorem to the case of $K$ and $L$ being skew fields or associative rings.

Continuation of this research was made in papers by E.I.\,Bunina (\cite{Bun1}--\cite{Bun4}, 1998--2009), in which the results of A.I\,Maltsev were extended for unitary linear groups over skew fields and associative rings with involutions, and also for Chevalley groups over fields and local rings.

In 2000, V.Tolstikh considered in \cite{Tolstyh} the connection of the second order properties of skew fields with the first order properties of automorphism groups of spaces of infinite dimension over these skew fields. In 2003, E.I.~Bunina and A.V.~Mikhalev considered the connection of the second order properties of associative rings and the first order properties of categories of modules, endomorphism rings, automorphism groups and projective spaces of modules of infinite rank over these rings (see~\cite{categories}).

In~\cite{Abelian}, E.I.~Bunina and A.V.~Mikhalev discovered connection of second order properties of an Abelian $p$-group with first order properties of its endomorphism ring (the analogue of Baer--Kaplansky Theorem for elementary equivalence).

In~\cite{My1}, E.I.~Bunina and M.A.~Roizner discovered connection of first order properties of the automorphism group of an Abelian $p$-group with second order properties of the divisible part and the basic subgroup of the group.

This paper continues the paper~\cite{My1}. We discover connection of first  order properties of the automorphism group of an Abelian $p$-group with second order properties of the group bounded by its final rank provided that the group is reduced and $p>2$.

\section{Background}

It is said that an~element~$a\in A$ is \emph{divisible by}
a~positive integer~$n$ (denoted as $n\mid a$)
if there is an element $x\in A$ such that $nx=a$. A~group~$D$ is called \emph{divisible} if $n\mid a$ for all $a\in D$ and all natural~$n$.
The groups $\mathbb Q$ and $\mathbb Z(p^\infty)$ are examples of divisible groups. A~group~$A$ is called \emph{reduced} if it has no nonzero divisible subgroups.

A~subgroup $G$ of a~group~$A$ is called \emph{pure} if the equation $nx=g\in G$ is
solvable in~$G$ whenever it is solvable in the whole group~$A$.
In other words, $G$~is pure if and only if
$$
\forall n\in \mathbb Z\quad nG=G\cap nA.
$$

A~subgroup~$B$ of a~group~$A$ is called
a~$p$-\emph{basic subgroup} if it satisfies the following constraints:
\begin{enumerate}
\item
$B$ is a~direct sum of cyclic $p$-groups and infinite
cyclic groups;
\item
$B$ is pure in~$A$;
\item
$A/B$ is $p$-divisible.
\end{enumerate}

Every group, for every prime~$p$, contains $p$-basic
subgroups~(\cite{Fuks}).

We now focus on $p$-groups, where $p$-basic subgroups are particularly
important. If $A$ is a~$p$-group and $q$ is a~prime different from~$p$
then, evidently, $A$ has only one $q$-basic subgroup, namely~$0$.
Therefore, in $p$-groups we may refer to the $p$-basic subgroups
simply as \emph{basic} subgroups without confusion.

We need the following facts about basic subgroups.

\begin{theorem}[\cite{Sele7}]\label{BasicMaxBounded}
Assume that $B$ is a~subgroup of a~$p$-group~$A$,
$B=\bigoplus\limits_{n=1}^\infty B_n$, and $B_n$~is a~direct
sum of groups~$\mathbb Z(p^n)$.
Then $B$~is a~basic subgroup of~$A$ if and only if
for every integer $n> 0$, the subgroup $B_1\oplus \dots\oplus B_n$
is a~maximal $p^n$-bounded direct summand of~$A$.
\end{theorem}

\begin{theorem}[\cite{Fuks}]\label{BasicEndIm}
A basic subgroup of a~$p$-group~$A$ is an endomorphic image of the group~$A$.
\end{theorem}

An infinite system $L=\{ a_i\}_{i\in I}$ of elements of the group~$A$ is called \emph{independent} if every finite subsystem of~$L$ is independent. An independent system~$M$ of~$A$ is \emph{maximal} if there is no independent system in~$A$ containing~$M$ properly. By the \emph{rank} $r(A)$ of a~group~$A$ we mean the cardinality of a~maximal independent system containing only elements of infinite and prime power orders.
The \emph{final rank} of a~basic subgroup~$B$ of a~$p$-group~$A$ is the
infimum of the cardinals $r(p^nB)$.

In the paper~\cite{My1}, E.I.~Bunina and M.A.~Roizner introduced certain formulas for operating with involutions (i.\,e.~automorphisms of order~$2$). The declarations of these formulas follow below.

An involution~$\eps$ corresponds to the decomposition of the group~$A$ into direct sum $A = A_\eps^+ \oplus A_\eps^-$, where $A_\eps^+=\{ a\in A\mid \eps a=a\}$ and $A_\eps^-=\{ a\in A\mid \eps a=-a\}$.

Formula~$Extreme(\eps)$ means that the automorphism~$\eps$ is an extreme involution (i.\,e. an involution which has one of its summands~$A_\eps^+$ and $A_\eps^-$ as indecomposable).
The indecomposable summand for the~involution~$\eps$ is denoted by~$A_\eps$, while the other summand is denoted by~$A_\eps^\perp$.

With only an~involution~$\xi$, one cannot distinguish the groups$A_\xi^+$ and $A_\xi^-$ in the first order language. Therefore we consider pairs~$(\xi, \eps)$ with the condition~$Extreme(\eps)\land\xi\eps = \eps\xi$. For such pairs, either $A_\eps \subset A_\xi^+$ or $A_\eps \subset A_\xi^-$. $A_\eps$ indicates the required group among~$A_\xi^+$ and~$A_\xi^-$ (it is denoted by~$A_{(\xi, \eps)}$). The property of being a pair is denoted by the formula~$Pair(\xi, \eps) \define \xi^2 = 1 \land Extreme(\eps) \land \xi\eps = \eps\xi$. Instead of $\forall \xi \forall \eps (Pair(\xi, \eps) \then (\dots))$ and $\exists \xi \exists \eps (Pair(\xi, \eps) \land (\dots))$, we will write $\forall (\xi, \eps)$ and $\exists(\xi, \eps)$ respectively.

The following formulas dealing with involutions, extreme involutions and involution pairs, were defined in the paper~\cite{My1}~(p.\,7--8, 10, 24):

1) $\eps \in \eps_1 \oplus \eps_2$ iff $A_\eps \subset A_{\eps_1}
\oplus A_{\eps_2}$ and $A_\eps^\perp\supset A_{\eps_1}^\perp \cap
A_{\eps_2}^\perp$ for extreme involutions~$\eps$, $\eps_1$, $\eps_2$ such that $\eps_1
\eps_2 = \eps_2 \eps_1$;

2) $ \eps_2 \in (\xi_1, \eps_1)$ iff $A_{\eps_2} \subset A_{(\xi_2, \eps_2)}$ for an~extreme involution~$\eps_2$ and a~pair~$(\xi_1, \eps_1)$;

3) $(\xi_1, \eps_1) \subset (\xi_2, \eps_2)$ iff $A_{(\xi_1, \eps_1)} \subset A_{(\xi_2, \eps_2)}$;

4) $(\xi_1, \eps_1) = (\xi_2, \eps_2)$ iff $A_{(\xi_1, \eps_1)} = A_{(\xi_2, \eps_2)}$;

5) $(\xi_1, \eps_1) \cap (\xi_2, \eps_2) = (\xi_3, \eps_3)$ iff $A_{(\xi_3, \eps_3)} = A_{(\xi_1, \eps_1)} \cap A_{(\xi_2, \eps_2)}$;

6) $(\xi_1, \eps_1) \oplus (\xi_2, \eps_2) = (\xi_3, \eps_3)$ iff $A_{(\xi_3,\eps_3)} = A_{(\xi_1, \eps_1)} \oplus A_{(\xi_2, \eps_2)}$;

7) $\overline{(\xi_1, \eps_1)} = (\xi_2, \eps_2)$ iff $A_{(\xi_1, \eps_1)} \oplus A_{(\xi_2, \eps_2)} = A$;

8) the formula $f(\eps_1) = \eps_2$ for extreme involutions~$\eps_1, \eps_2$ and an automorphism~$f$ means that~$f(A_{\eps_1}) = A_{\eps_2}$. But since this situations is possible if only the summands~$A_{\eps_1}$ and~$A_{\eps_2}$ have equals orders it is convenient to define another formala for matching summands which have different orders:
\begin{multline*}
    \eps_1 \mapson{f} \eps_2 \define
        Extreme(\eps_1) \land
        Extreme(\eps_2) \land
        f(\eps_1) \in \eps_1 \oplus \eps_2 \land
        f(\eps_1) \ne \eps_1 \land
        f(\eps_1) \ne \eps_2;
\end{multline*}

9) the formula $ord (\eps_1) < ord (\eps_2)$ means that the order of an~involution~$\eps_1$ (i.\,e. the order of the corresponding summand~$A_{\eps_1}$) is less than the order of an~involution~$\eps_2$. Similarly, all the other order relations can be defined.

\section{Specifying basic subgroup}

Let $A$ be an~unbounded reduced Abelian $p$-group with cardinality~$\mu$ and final rank~$\mu_{fin}$ of~the basic subgroup. There exists a decomposition~$A = A_1 \oplus A_2$ such that the order of any indecomposable subgroup of the group~$A_1$ is less than the order of any indecomposable subgroup of the group~$A_2$ and the basic subgroup of $A_2$ has rank~$\mu_{fin}$.
The formula specifying this decomposition follows:

\begin{lemma}\label{Lemma1}
Define a~formula:
\[
    ByOrd\Pair \define
        \forall \eps' \Big(
            Extreme(\eps') \then \big(
                \eps' \in \Pair \iff
                    ord(\eps') \geq ord(\eps)
            \big)
        \Big).
\]

The~formula
\begin{multline*}
    Final\Pair[0] \define
        ByOrd\Pair[0] \land
        \forall \Pair[1] \subsetneq \Pair[0] \bigg(
            ByOrd\Pair[1] \then\\
            \then \exists f \Big(
                \big(
                    \forall \eps \in \overline{\Pair[1]}\quad
                    f(\eps) = \eps
                \big) \land
                \big(
                    \forall \eps \in \Pair[0] \cap \overline{\Pair[1]}\quad
                        \exists \eps' \in \Pair[1]\quad
                            \eps' \mapson{f} \eps
                \big)
            \Big)
        \bigg)
\end{multline*}
specifies the decomposition~$A = A_1 \oplus A_2$, where $A_1 = A_{\overline{\Pair[0]}}$, $A_2 = A_{\Pair[0]}$.
\end{lemma}
\begin{proof}
The formula~$ByOrd$ specifies such decompositions~$A = A_{\overline \Pair} \oplus A_{\Pair}$ that the order of any indecomposable subgroup of the group~$A_{\overline \Pair}$ is less than the order of any indecomposable subgroup of the group~$A_{\Pair}$. These decompositions will be referred as order-decompositions. The~formula~$Final$ states that, first, the decomposition~$A = A_{\overline{\Pair[0]}} \oplus A_{\Pair[0]}$ is an order-decomposition and, second, for any order-decomposition~$A_{\Pair[0]} = A_{\Pair[1]} \oplus A_{\Pair[0] \cap \overline{\Pair[1]}}$ the rank of the group~$A_{\Pair[0] \cap \overline{\Pair[1]}}$ is less than or equal to the rank of the basic subgroup of the group~$A_{\Pair[1]}$. The latter means that the rank of the basic subgroup of~$A_{\Pair[0]}$ equals to the final rank~$\mu_{fin}$.
\end{proof}

Fix the pair~$\Pair[0]$, and let~$A_{low} = A_{\overline{\Pair[0]}}$, $A_{fin} = A_{\Pair[0]}$.

\begin{lemma}\label{Lemma2}
For an unbounded reduced Abelian $p$-group~$A$, there exists an automorphism~$\nu$ such that~$\nu\,\big|_{A_{low}} = \id$ and, for some basic subgroup~$B$, $\im \left(\nu - \id\,\big|_{A_{fin}}\right) = B$, i.\,e. for any~$b \in B$ there exists~$a \in A_{fin}$ such that~$\nu(a) = a + b$ and, vice versa,  for any~$a \in A_{fin}$, $\nu(a) = a + b$ for some~$b \in B$.
\end{lemma}
\begin{proof}
There exists a basic subgroup~$B$ such that~$A_{low} \subset B$. By Theorem~\ref{BasicEndIm}, there exists such endomorphism~$\eps\colon\,A \to B$ that~$\eps\,\big|_{A_{low}} = \id$ and~$\im \left(\eps\,\big|_{A_{fin}} \right) = B\cap A_{fin}$, and for all~$a \in A_{fin}$ $,\ord(\eps(a)) < \ord(a)$. We define the automorphism~$\nu$ on~$A_{low}$ and~$A_{fin}$ independently in the following way: $\nu\,\big|_{A_{low}} = \id$ and $\nu\,\big|_{A_{fin}} = \id + \eps$. Clearly, it is the required automorphism.
\end{proof}

We associate each automorphism~$\nu$ with a~subgroup~$B_\nu$ with the following formula. The formula indicates if the indecomposable subgroup for an extreme involution~$\eps$ lies in~$B_\nu$:
\begin{multline*}
    InBase(\eps, \nu) \define
        \exists\,\eps_{low}, \eps_{fin} \Big(
            Extreme(\eps_{low}) \land
            Extreme(\eps_{fin}) \land
            \eps \in \eps_{low} \oplus \eps_{fin} \land\\
            \land
            \eps_{low} \subset A_{low} \land
            \eps_{fin} \subset A_{fin} \land
            \exists\,\eps'\subset A_{fin} \big(
                \eps' \mapson{\nu} \eps_{fin} \land
                \ord(\eps') > \ord(\eps_{fin})
            \big)
        \Big)
\end{multline*}

\begin{lemma}\label{Lemma3}
For the automorphism~$\nu$ defined in Lemma~\ref{Lemma2}, the corresponding subgroup~$B_\nu$ coincide with the original basic subgroup~$B$.
\end{lemma}
\begin{proof}
Let $A_\eps$ lie in~$B$. Then~$A_\eps$ lies in the direct sum~$A_{\eps_{low}} \oplus A_{\eps_{fin}}$, where $A_{\eps_{low}} \subset B \cap A_{low} =: B_{low}$ and $A_{\eps_{fin}} \subset B \cap A_{fin} =: B_{fin}$. Let $A_{\eps_{fin}} = \langle b \rangle$. Then there exists an element~$a \in A_{fin}$ with greater order than $b$ such that $\nu(a) = a + b$. Then the extreme involution corresponding to the indecomposable subgroup~$\langle a \rangle$ can be chosen as~$\eps'$ since $\langle a \rangle \oplus \nu(\langle a \rangle) = \langle a \rangle \oplus A_{\eps_{fin}}$.

In reverse, let the formula~$InBase$ hold for an~extreme involution~$\eps$. Then  $A_{\eps_{low}}$ clearly lies in~$B$. Consider~$A_{\eps_{fin}} = \langle b \rangle$. Let $A_{\eps'} = \langle a \rangle$, $\nu(a) = ka + lb$. Be the construction of~$\nu$, $k = 1$. Since $\langle a \rangle \oplus \langle ka + lb \rangle = \langle a \rangle \oplus \langle b \rangle$, $l{\not\vdots} p$. Then for some~$m$, $\nu(ma) = ma + b$, i.\,e. $b \in B$. The statement is proved.
\end{proof}

Now we need to write the requirement on~$\nu$ stating that~$B_\nu$ is basic. We introduce some formulas.

1. The formula
\[
    Rest_\eps(\xi_1, \eps_1) \define
        \forall \eps_2 \big(
            \eps_2 \in (\xi_1, \eps_1) \then
                ord(\eps_2) \leq ord(\eps)
        \big)
\]
selects involution pairs~$(\xi_1, \eps_1)$ which correspond to direct sums of cyclic groups with order at most~$ord(A_\eps)$.

2. The formula
\begin{multline*}
    MaxRest_{\eps}(\xi_1, \eps_1) \define
        Rest_\eps(\xi_1, \eps_1)\land
        \forall (\xi_2, \eps_2) \big(
            Rest_\eps(\xi_2, \eps_2) \then
                (\xi_1, \eps_1) \not \subset (\xi_2, \eps_2)
        \big)
\end{multline*}
selects involution pairs~$(\xi_1, \eps_1)$ which correspond to maximal direct sums of cyclic groups with order at most~$ord(A_\eps)$.

\begin{lemma}\label{Lemma4}
For an automorphism~$\nu$, the formula
\begin{multline*}
IsBase(\nu) \define \forall \eps_0\ Extreme(\eps_0) \then \forall (\xi, \eps) \Big( \forall \eps' \big(\eps' \subset (\xi, \eps) \iff\\
\iff \ord(\eps') \leq \ord(\eps_0) \land InBase(\eps', \nu)\big) \then MaxRest_{\eps_0}(\xi, \eps)\Big)
\end{multline*}
is true if and only if the subgroup~$B_\nu$ is basic.
\end{lemma}
\begin{proof}
The requirement~$IsBase(\nu)$ means that each limitation of the subgroup~$B_\nu$ with the order~$\ord(\eps_0)$ is a maximal $\ord(\eps_0)$-bounded summand of the group~$A$. The statement of the lemma is implied from Theorem~\ref{BasicMaxBounded}.
\end{proof}

\section{Specifying definable sets in basic subgroup}

In the paper~\cite{My1}, a variant of Shelah theorem (\cite{Shelah}) was proved for the case when  $\Omega$ is the set of automorphism tuples encoding endomorphisms of the group~$A=\bigoplus\limits_\mu \Z{p^l}\ (l \in \mathbb N)$.

\begin{theorem}\label{ShelahOriginal}
There exists a formula~$\widetilde \varphi(\dots)$ satisfying the following statement.
Let $\{ f_i\}_{i\in \mu}$ be a set of elements from $\Omega$. Then there exists a vector $\overline g$ such that the formula~$\widetilde
\varphi( f,\overline g)$
is true in $\Omega$ if and only if $f=f_i$ for some~$i\in \mu$.
\end{theorem}

We need to interpret mappings of a set of extreme involutions from the basic subgroup~$B$ into itself in order to use Shelah theorem for the case of indecomposable direct summands of~$B$. For this purpose, accordingly to the previous section, we construct for a mapping~$f$ two automorphisms $f_1$~and~$f_2$, which correspond to~$B$, and define
$$
    f(A_{\eps_1}) = A_{\eps_2} \define
        \exists \eps \big(
            Extreme(\eps) \land
            \eps \mapson{f_1} \eps_1 \land
            \eps \mapson{f_2} \eps_2
        \big).
$$

A composition of such mappings can be easily expressed with the latter formula.

Hence we get Shelah Theorem in the following formulation.

\begin{theorem}\label{Shelah}
Let $\Omega$ be the set of extreme involutions corresponding to dircet summands from~$B$.
There exists a formula~$\widetilde \varphi(\dots)$ satisfying the following statement.
Let $\{ f_i\}_{i\in \mu}$ be a~set of elements from $\Omega$. Then there exists a vector $\overline g$ such that the formula~$\widetilde \varphi( f,\overline g)$ is true in $\Omega$ if and only if $f=f_i$ for some~$i\in \mu$.
\end{theorem}

\section{Structuring basic subgroup}

By Theorem~\ref{Shelah}, we define a set of extreme involutions that corresponds to decomposition of basic subgroups into indecomposable summands. This set must satisfy two conditions: first, involutions in it must be independent of each other, and second, any superset of extreme involutions must have dependent involutions. Denote this set by~$\F_B$.

Let $B = \bigoplus\limits_{i} B_i$, where $B_i$ are indecomposable summands. We are to define a set of automorphisms~$g_{ij}$ with $\ord(B_i) \geq \ord(B_j)$ which specify generators~$b_i$ of these indecomposable summands. Precisely, $B_i = \langle b_i \rangle$ for each~$i$, $g_{ij} \Big|_{\bigoplus\limits_{m \ne i} B_m} = \id$ and $g_{ij}(b_i) = b_i + b_j$ for each~$i, j$.

We need the following technical lemma.

\begin{lemma}\label{Lemma5}
Let $g$ be such an automorphism that $g \Big|_{\bigoplus\limits_{m \ne i} B_m} = \id$ and $g(b_i) = k_1 b_i + k_2 b_j$ for some $i, j, k_1 \ne 0, k_2 \ne 0$, where $B_i = \langle b_i \rangle$, $B_j = \langle b_j \rangle$. Then $k_1 = 1$ if and only if the constraint $g_0^{-1} g g_0(B_k) \subset B_k \oplus B_j$ is true for some $k$ and an~automorphism~$g_0$ with $g_0 \Big|_{\bigoplus\limits_{m \ne k} B_m} = \id$ and $g_0(b_k) = l_1 b_k + l_2 b_i$ for some $l_1 \ne 0, l_2 \ne 0$, where $B_k = \langle b_k \rangle$.
\end{lemma}
\begin{proof}
$$
    g_0^{-1} g g_0(b_k) =
    g_0^{-1} g(l_1 b_k + l_2 b_i) =
    g_0^{-1}(l_1 b_k + l_2 k_1 b_i + l_2 k_2 b_j) =
    b_k + l_2 (k_1 - 1) b_i + l_2 k_2 b_j.
$$
Hence, it is clear that $g_0^{-1} g g_0(B_k)$ is in $B_k \oplus B_j$ if and only if $k_1 = 1$.
\end{proof}

Now we specify constraints for the set of automorphisms.

\begin{enumerate}
\item
For each $i$ and $j$, $\ord(B_i) \geq \ord(B_j)$, there is exactly one automorphism~$g_{ij}$ in the set, which is identical on~$\bigoplus\limits_{m \ne i} B_m$ and maps $B_i$ into a subgroup of~$B_i \oplus B_j$ that is equal neither to $B_i$ nor to $B_j$. There must be no other automorphisms in the set.

\item
For each automorphism~$g_{ij}$ from the set, $g_{ij}(b_i) = b_i + k_{ij} b_j$ for some~$k_{ij}$, where $B_i = \langle b_i \rangle$, $B_j = \langle b_j \rangle$. This constraint can be expressed in a formula due to Lemma~\ref{Lemma5}.

\item
For any three automorphisms~$g_{ij}, g_{jk}, g_{ik}$ from this set, $g_{jk}^{-1} g_{ij}^{-1} g_{jk} g_{ij} = g_{ik}$. This constraint adjusts the coefficients~$k_{ij}$ with each other. Hence, it can be assumed that the coefficients~$k_{ij}$ are equal to~$1$ (by choosing the corresponding generators~$b_i$), i.\,e. $g_{ij}(b_i) = b_i + b_j$.
\end{enumerate}

We denote this set, which is provided by Theorem~\ref{Shelah}, by~$\F_g$.

\section{Interpretation of the first order logic of the group~$A$}

In this section, we express the first order logic of the group~$A$ in terms of the first order language of its automorphism group. For this purpose, it is sufficient to interpret each element of the group by some automorphism and to define the formulas for equality and addition of two elements. Then any statement in the first order language of the Abelian group~$A$ can be translated into an equivalent statement in the first order language of the automorphism group~$\Aut A$ by replacing all the quantifiers over elements of the group and all the predicates of equality and addition with the corresponding quantifiers over the interpreting automorphisms and the formulas for equality and addition. (For the details of the translation, see paper~\cite{My1}.)

Notice that each element of the group~$A$ has finite order. Thus, there exists a decomposable direct summand~$B_i = \langle b_i \rangle$ of the subgroup~$B$ which has greater order. Then there exists an automorphism~$f$ which is identical on the direct complement to~$B_i$ and which maps~$b_i$ to~$b_i + a$. It is the automorphism which encodes~$a$.

The formula for selecting such automorphisms~$f$ is the following:
$$
    \exists \eps_i \in \F_B\
        \exists \eps_a \Big(
            Extreme(\eps_a) \land
            \eps_i \mapson{f} \eps_a
        \Big) \land
        \forall \eps_j \in \F_B \Big(
            \eps_i \ne \eps_j \then
                f(\eps_j) = \eps_j
        \Big).
$$

Now here is the formula for equality of two such automorphisms~$f_1$ and~$f_2$:
\begin{multline*}
    f_1 \doteq f_2 \define
        \exists \eps_1, \eps_2 \in \F_B\
            \exists \eps_a \Big(
                Extreme(\eps_a) \land
                \eps_1 \mapson{f_1} \eps_a \land
                \eps_2 \mapson{f_2} \eps_a \land\\
                \land \exists g_{ij} \in \F_g \big(
                    \eps_1 \mapson{g_{ij}} \eps_2 \land
                    g_{ij}^{-1} f_2 g_{ij} = f_1
                \big)
            \Big).
\end{multline*}

Finally, here is the formula for addition of such automorphisms:
\begin{multline*}
    f_1 \dotplus f_2 \doteq f_3 \define
        \exists f_1^\prime, f_2^\prime, f_3^\prime \Big(
            f_1 \simeq f_1^\prime \land
            f_2 \simeq f_2^\prime \land
            f_3 \simeq f_3^\prime \land\\
            \land
            \exists \eps_i \in \F_B \big(
                f_1^\prime(\eps_i) \ne \eps_i \land
                f_2^\prime(\eps_i) \ne \eps_i \land
                f_3^\prime(\eps_i) \ne \eps_i \land
                f_3^\prime = f_1^\prime f_2^\prime
            \big)
        \Big).
\end{multline*}

These formulas provide interpretation of the first order logic of the group~$A$.

\section{Interpretation of the second order logic of the group~$A$}

Recall, we are concerned with an unbounded reduced Abelian $p$-group~$A$ which has decomposition~$A = A_1 \oplus A_2 = A_{low} \oplus A_{fin}$, where~$A_{fin}$ has rank of basic subgroup equal to the final rank~$\mu_{fin}$. There is also the decomposition of the basic subgroup~$B = B_{low} \oplus B_{fin}$, where $B_{low} \subset A_{low}$, $B_{fin} \subset A_{fin}$. In this section, we express the second order logic, bounded with~$\mu_{fin}$, of the group~$A$. The idea of interpretation is the same as one in the paper~\cite{My1}.

We need a set of independent involution pairs where each pair corresponds to a direct summand of the group~$B_{fin}$. Each such direct summand must contain indecomposable direct summands of arbitrary big order. There must be total of~$\mu_{fin}$ such pairs. This set can be defined by Theorem~\ref{Shelah} in the same way as in the paper~\cite{My1}. Denote this set by~$\F_{fin}$.

On each direct summand corresponding to a pair~$\Pair$ from~$\F_{fin}$, we interpet an element of the group~$A$ with an automorphism in the same way as in the previous section with the only difference that indecomposable summands from~$A_{\Pair}$ should be used instead of indecomposable summands from~$B$. Two such automorphism are equivalent (i.\,e. encode the same group element) if they differ by an automorphism that is identical on~$A_{\Pair}$:
$$
f_1 \sim_{\Pair} f_2 \define
    \exists f \big(
        \forall \eps \in \Pair
            f(\eps) = \eps \land
        f_1 = f_2 f
    \big).
$$

The remaining part of the proof of expressibility of the second order logic is totally similar to one in the paper~\cite{My1}. Expressly, each sentence~$\phi$ of the bounded second order logic of the group~$A$ has a corresponding sentence~$\psi$ in the first order logic of the group~$\Aut A$ which is constructed by the certain algorithm. In this algorithm, all the object variables are replaced with the encoding automorphisms and all the $k$-ary predicates are replaced with $k$ automorphisms~$f_1, \dots, f_k$ which encode elements on each direct summand~$A_{\Pair}, \Pair \in \F_{fin}$. A~tuple~$(x_1, \dots, x_k)$ belongs to this predicate whenever there is a direct summand~$A_{\Pair}$ on which the automorphism~$f_i$ encodes the element~$x_i$, for each $i = 1, \dots, k$. This algorithm is described in details in the papers~\cite{Abelian},~\cite{My1}.

\end{document}